\documentclass{amsart}
\usepackage{amsmath,amsthm,amssymb}
\newcommand{\F}{{\mathbb F}}
\newcommand{\Fp}{{{\F}_p}}
\renewcommand{\ker}{\operatorname{Ker}\nolimits}

\newenvironment{psmatrix}{\left(\begin{smallmatrix}}{\end{smallmatrix}\right)}
\def\Matrix#1#2#3#4{\left(\begin{matrix}#1&#2\\#3&#4\end{matrix}\right)}
\def\psMatrix#1#2#3#4{\begin{psmatrix}#1&#2\\#3&#4\end{psmatrix}}
\newtheorem{Thm}{Theorem}[section]
\newtheorem{Prop}[Thm]{Proposition}
\newtheorem{Lem}[Thm]{Lemma}
\title[Generalized invariants of reflection groups]{Generalized invariants of reflection groups in rank two}
\author{Jaume Aguad\'{e}}
\email{Jaume.Aguade@uab.cat}
\address{Departament de Matem\`atiques, Universitat Aut\`onoma de Barcelona, 08193 Cerdanyola del Vall\`es, Spain.}
\thanks{The author is partially supported by grants MTM2013-42293-P and 2009SGR-1092.}
\keywords{Reflection group, polynomial invariants, generalized invariants, stable invariants}
\begin{document}
\begin{abstract}
For each subgroup of $GL_2(\Fp)$ or order divisible by $p$, generated by (pseudo-)reflections, we compute the ideals of 
stable and generalized invariants. These groups and these ideals are related to the cohomology of compact Lie 
groups, Kac-Moody groups and $p$-compact groups.
\end{abstract}

\maketitle

%
%
%
%
%
%
%
%

\section{Introduction}

The invariant theory of finite reflection groups (also called pseudoreflection groups) in positive characteristic 
dividing the order of the group displays a series of phenomena which do not appear when the order of the group is a unit 
in the base field. Some of these phenomena are not yet well understood. Generalized invariants and stable invariants are 
two of these phenomena. They were defined by Kac in 1985 (\cite{Kac-Lie}) based on work of Demazure in 1973 
(\cite{Demazure}), to provide a purely Weyl group explanation of the mod $p$ cohomology of the compact connected Lie 
groups. If the order of the reflection group $W$ is prime to $p$, then both generalized and stable invariants coincide 
with the standard polynomial invariants of $W$, but if this non-modularity condition fails, the generalized invariants 
and the stable invariants seem to be very hard to compute and they are only known in a few trivial cases. For instance, 
neither the generalized nor the stable invariants of the general linear group over a finite field are known and it has 
been conjectured that they are trivial, in the sense that they coincide with the ordinary invariants, i.e.\ 
the Dickson invariants.

In this paper we compute the generalized and the stable invariants of all reflection subgroups of $GL_2(\Fp)$ 
of order divisible by $p$. Our results support the conjecture above. In a final section, we review the appearance of 
these rank two reflection groups in topology, as mod $p$ reductions of Weyl groups of some families of topological 
groups, like Lie groups, Kac-Moody groups or $p$-compact groups.

%
%
%
%
%
%
%
%

\section{Preliminaries and notations}
Most of what we need in this paper can be found in the introduction to \cite{Larry-Ring}, but we include it here for 
completeness and in order to fix the notations that we will be using along the paper. The twin books \cite{Larry-book} 
and \cite{Benson} are excellent references for anything that is stated here without a proof.

An element $\sigma\in GL_n(\Fp)$ is a \emph{reflection} if $\sigma-1$ has rank $1$. Some authors call this a 
\emph{pseudoreflection} and reserve the word reflection for the case in which $\sigma$ is diagonalizable 
and $\sigma^2=1$. We call $x_1,\ldots,x_n$ the standard basis of the $\Fp$-vector space $\Fp^n$ and then we consider 
the linear action of $GL_n(\Fp)$ on the polynomial ring $P_n:=\Fp[x_1,\ldots,x_n]$. Some authors consider the dual 
action. If $G$ is any subgroup of $GL_n(\Fp)$, then the polynomials in $P_n$ invariant under all $g\in G$ form the ring 
of invariants $P_n^G$. If this ring of invariants is a polynomial ring, then the group $G$ is generated by reflections 
and $P_n$ is a free $P_n^G$-module of rank equal to $\lvert G\rvert$. If a group $G\subseteq GL_n(\Fp)$ is generated by 
reflections we say that $G$ is a \emph{reflection group}, but we have to keep in mind that being a reflection group is 
not an intrinsic property of a group, but a property of a representation of this group.

The ring of coinvariants of $G\subseteq GL_n(\Fp)$ is the quotient of $P_n$ by the ideal $J_1(G)$ generated by the elements in $P_n^G$ of positive degree.
$$(P_n)_G:=P_n/J_1(G)=\Fp\otimes_{P_n^G}P_n.$$
If $P_n^G$ is a polynomial algebra, then $(P_n)_G$ is an algebra with Poincar{\'e} duality.
The ring of coinvariants inherits an action of $G$ and if $\lvert G\rvert$ is prime to $p$, then $(P_n)_G$ coincides with the regular representation of $G$. In particular, $(P_n)_G$ has no invariants of positive degree. However, when $p$ divides the order of $G$, the ring $(P_n)_G$ can have invariants of positive degree and we van inductively define a chain of ideals of $P_n$  
$$J_1(G)\subseteq J_2(G)\subseteq\cdots$$ as follows:
$$A_0:=P_n,\quad A_i:=(A_{i-1})_G,\quad J_i(G):=\ker\{P_n\to A_{i}\}.$$
We can also consider the ideal $J_\infty(G):=\cup_iJ_i(G)$ which is called the ideal of \emph{stable invariants} of $G$. As we said before, if the order of $G$ is prime to $p$, all these ideals coincide with the ideal $J_1(G)$ of ordinary invariants of $G$. We also refer to the sequence $\{J_i(G)\}_{i=1}^\infty$ as the sequence of ideals of stable invariants of $G$.

The ring of invariants of a reflection group of order prime to $p$ is a polynomial algebra, but this fails if we do not assume that $\lvert G\rvert$ is prime to $p$. Hence, one can ask about which reflection groups in  $GL_n(\Fp)$ have polynomial rings of invariants. The work of Kemper-Malle (\cite{Kemper-Malle}) provides a complete list of all irreducible groups in $GL_n(\F_{p^r})$ with polynomial rings of invariants, for all values of $n$, $p$, $r$. 

Given a reflection $\sigma\in GL_n(\Fp)$, Demazure (\cite{Demazure}) and Kac-Peterson (\cite{Kac-Pet}, \cite{Kac-Lie}) 
defined an $\Fp$-linear operator $\Delta_\sigma:P_n\to P_n$ of degree $-1$ as follows. Chose a non zero vector 
$v_\sigma$ in the image of $\sigma-1$, then $$\Delta_\sigma(f):=\frac{\sigma(f)-f}{v_\sigma}.$$ These operators are 
well defined up to a non zero multiplicative constant. They behave as twisted derivations 
$$\Delta_\sigma(fg)=\Delta_\sigma(f)g+\sigma(f)\Delta_\sigma(g)$$ and play an important role in understanding the 
cohomology of classifying spaces of compact connected Lie groups (and, more in general, Kac-Moody groups) at torsion 
primes (see \cite{Kac-Lie} and \cite{Larry-Lie}). These operators were used by Kac-Peterson to define the \emph{ideal of 
generalized invariants} $I(\mathcal{S})$ of a set of reflections $\mathcal{S}\subset GL_n(\Fp)$. We say that a 
polynomial $f\in P_n$ of positive degree $k$ is a generalized invariant with respect to $\mathcal{S}$ if for any 
sequence $\{\sigma_1,\ldots,\sigma_k\}$ of elements of $\mathcal{S}$ we have that 
$$\Delta_{\sigma_1}\cdots\Delta_{\sigma_k}(f)=0.$$
It can be proved (\cite{Larry-Ring}, lemma 2.1) that the generalized invariants form an ideal $I(\mathcal{S})\subset P_n$. By definition, this ideal depends on the set $\mathcal{S}$ and not just on the reflection group $G(\mathcal{S})$ generated by $\mathcal{S}$. There is an interesting relation between generalized invariants and stable invariants (\cite{Larry-Ring}, lemma 3.2)
$$J_1(G(\mathcal{S}))\subseteq I(\mathcal{S})\subseteq J_\infty(G(\mathcal{S})).$$
In general, these three ideals can be different. If the reflections in $\mathcal{S}$ are diagonalizable, then 
$I(\mathcal{S})= J_\infty(G(\mathcal{S}))$. If the order of $G(\mathcal{S})$ is prime to $p$, then ordinary, stable and 
generalized invariants coincide. 

As said before, for reflection groups of order divisible by $p$ the ring of invariants may fail to be a polynomial ring. 
However, a very interesting property of the ideal of generalized invariants is given in the following theorem of 
Kac-Peterson (\cite{Kac-Pet}): $I(\mathcal{S})$ is always generated by a regular sequence of length $n$.

In this paper we will be dealing with the case of $GL_2(\Fp)$ and we use the following notations. $P$ will be the polynomial ring on two generators $P:=\Fp[x,y]$. $\zeta$ will denote a generator of $\Fp^*$. We denote $\omega,\omega'\in GL_2(\Fp)$ the reflections 
$$\omega:=\Matrix1011,\quad\omega':=\Matrix1101.$$ 
The operators associated to these two reflections (choosing $v_\omega=(0,1)$ and $v_{\omega'}=(1,0)$) are denoted 
$\Delta$ and $\overline\Delta$, respectively. Diagonal matrices will we denoted this way
$$D(a,b):=\Matrix{a}00b.$$
There are some polynomials in $P$ that will appear often and we fix here the notation that we use for them. 
\begin{align*}
\delta&:=xy^p-x^py\\
 d_1&:=\frac1\delta\,(xy^{p^2}-x^{p^2}y)\\
 d_0&:=\frac1\delta\,(x^py^{p^2}-x^{p^2}y^p)\\
 \gamma_i&:=x^{i(p-1)}-x^{p-1}y^{(i-1)(p-1)}+y^{i(p-1)},\quad i\ge 1.
\end{align*}
$d_0$, $d_1$ are the well known Dickson polynomials and the polynomials $\gamma_i$ will appear in the computation of the stable invariants of the group $SL_2(\Fp)$.

Since we are only considering homogeneous polynomials, we sometimes simplify the notation by substituting the exponents of one of the variables by an asterisk. For instance, we can write $\gamma_i=x^{i(p-1)}-x^{p-1}y^{*}+y^{*}$.

%
%
%
%
%
%
%
%
%

\section{Reflection groups in dimension two and their invariants}
The well known classification of subgroups of $GL_2(k)$ for a finite field $k$ (see \cite{Suzuki}, 3.6) makes it easy to determine all reflection subgroups of $GL_2(\Fp)$. Let us introduce the following subgroups of $GL_2(\Fp)$.
\begin{itemize}
 \item For each $r|p-1$ let $L^r$ be the subgroup of all matrices $A\in GL_2(\Fp)$ such that $A^r$ has determinant 1. If $r=1$, this is the special linear group $SL_2(\Fp)$; if $r=p-1$, this is the full linear group $GL_2(\Fp)$. $L^r$ is a reflection group of order $rp(p^2-1)$ generated by the reflections $\omega$, $\omega'$, $D(\zeta^{\frac{p-1}r},1)$.
 \item For each $r,s|p-1$ let $U_{r,s}$ be the subgroup of all upper triangular matrices 
 $$\Matrix{\alpha}{\lambda}{0}{\beta}\in GL_2(\Fp)$$ such that $\alpha^r=\beta^s=1$. It is a reducible reflection group of order $rsp$ generated by the reflections $\omega$, $D(\zeta^{\frac{p-1}r},1)$, $D(1,\zeta^{\frac{p-1}s})$. 
 $U_{1,1}$ is the cyclic group of order $p$ in its representation as a Sylow subgroup of $GL_2(\Fp)$. As abstract groups, both $U_{2,1}$ and $U_{1,2}$ are isomorphic to the dihedral group of order $2p$, but they are not conjugate. These groups are considered in \cite{Larry-book} (pages 128--129), \cite{Larry-Ring} and also in \cite{ABKS} and \cite{AR} where they appeared as mod $p$ reductions of the (infinite) Weyl groups of some Kac-Moody groups. 
\end{itemize}
\begin{Thm}\label{grups}
 Let $G\subseteq GL_2(\Fp)$ be a reflection group with $\lvert G\rvert\equiv0\,(p)$. Up to conjugation, 
 \begin{enumerate}
  \item $G=U_{r,s}$ for some $r,s|p-1$, or
  \item $G=L^r$ for some $r|p-1$.
 \end{enumerate}
\end{Thm}

\begin{proof}
The result is straightforward for $p=2$ so we can assume $p>2$. According to Theorem 8.1 in \cite{Kemper-Malle},
the table 8.2 in \cite{Kemper-Malle} lists all irreducible primitive reflection groups in dimension two. The first 
entry in the list consists of groups $G$ in characteristic $p$ such that $SL_2(\Fp)\triangleleft G$. These are the 
groups $L^r$. The second entry consists of groups $G$ in characteristic $3$ such that $SL_2({\F}_5)\triangleleft G$. 
Since $\lvert SL_2({\F}_5)\rvert=120$, none of these groups can be a subgroup of $GL_2({\F}_3)$. All other groups in the 
list are representations over finite fields of the finite unitary reflection groups of dimension two in the list of 
Shephard-Todd (\cite{She-Todd}). All these groups have order prime to $p$. Hence, 
we only need to check the imprimitive groups and the reducible groups.
 
 Let $\theta_m$ denote a primitive $m$-th root of unity. The imprimitive reflection groups are the groups $$G(m,r):=\left\langle\Matrix0110,\;\Matrix{\theta_m^i}00{\theta_m^j},\;i+j\equiv0\,(r)\right\rangle$$ for any $m$ and any $r|m$. This group has order $2m^2/r$. The orthogonal groups $O^+_2(\Fp)$ and $O^-_2(\Fp)$ are particular cases of these groups. 
 Notice that $D(\theta_m,\theta_m^{-1})\in G(m,r)$ hence, if the group $G(m,r)$ can be represented in $GL_2(\Fp)$ then $\theta_m+\theta_m^{-1}\in\Fp$. Since $\Fp(\theta_m)$ is an extension of $\Fp(\theta_m+\theta_m^{-1})$ of degree $\le2$, if $G(m,r)$ can be represented in $GL_2(\Fp)$, then $\theta_m\in{\F}_{p^2}$ and $m|p^2-1$. This shows that all imprimitive reflection groups in $GL_2(\Fp)$ have order prime to $p$.
 
 Let $G$ be a reducible reflection subgroup of $GL_2(\Fp)$ and assume that $G$ is a subgroup of the group of upper triangular matrices in $GL_2(\Fp)$. 
 
 If $\psMatrix\alpha\lambda0\beta$ is a reflection, then $\alpha=1$ or $\beta=1$. Since $p$ divides the order of $G$, all upper unitriangular matrices are in $G$. Observe that 
 \begin{align*}
 \Matrix1{-\lambda/\beta}01\Matrix1\lambda0\beta&=\Matrix100\beta\\ 
 \Matrix1{-\mu}01\Matrix\alpha\mu01&=\Matrix\alpha001.
 \end{align*}
 Hence, if $G$ is generated by reflections
 $$G=\left\langle\Matrix1{\lambda_i}0{\beta_i},\Matrix{\alpha_j}{\mu_j}01,\;i=1,\ldots,n,\,j=1,\ldots,m\right\rangle$$
 then $D(\alpha_i,1)$, $D(1,\beta_j)$, $\omega'\in G,\;i=1,\ldots,n,\,j=1,\ldots,m$.
 Let $r$ be the order of the subgroup of $\Fp^*$ generated by $\alpha_1,\ldots,\alpha_n$ and let $s$ be the subgroup of $\Fp^*$ generated by $\beta_1,\ldots,\beta_m$. Then, it is clear that $G=U_{r,s}$.
\end{proof}

The ring of invariants of the general linear group is the well known Dickson Algebra (\cite{Benson}, \cite{Larry-book}). In the case of dimension two, $$P^{GL_2(\Fp)}=\Fp[d_1,d_0],\quad |d_1|=p^2-p,\,|d_0|=p^2-1.$$ 
Notice that $d_0=\delta^{p-1}$.
The invariants of the groups $L^r$ for $r|p-1$ are also well known: $$P^{L^r}=\Fp[d_1,\delta^r].$$

 The invariants of the reducible reflection groups $U_{r,s}$ are as follows.

\begin{Prop}\label{invariantsU}
 The invariants of $U_{r,s}$ form a polynomial algebra with generators $x^r$ and $(y^p-x^{p-1}y)^s$.
\end{Prop}

\begin{proof}
 It is trivial to check that $x^r$ and $(y^p-x^{p-1}y)^s$ are invariant under the generators of $U_{r,s}$. Since the Jacobian of these two polynomials does not vanish, we can use proposition 16 in \cite{Kemper} to conclude that $$P^{U_{r,s}}=\Fp[x^r,(y^p-x^{p-1}y)^s].$$ 
\end{proof}

%
%
%
%
%
%
%
%

\section{Bases for rings of coinvariants}

In this section we obtain explicit bases of $P$ over $P^G$ for each reflection group $G\subseteq GL_2(\Fp)$. 
We first state some useful identities involving the polynomials $d_1,d_0,\delta$, which will play an important role in 
the computations in this paper. The proofs are easy, either by a direct computation or using induction, and we omit the 
details.

\begin{Prop}\label{formules}Assume $p>2$. Then
 \begin{enumerate}
  \item $d_1=\sum_{i=0}^p\,x^{(p-1)i}y^{(p-1)(p-i)}$.
  \item $y^{p^2-1}=y^{p-1}d_1-d_0$.
  \item $y^{p^2-p+r}=y^rd_1-x^{p-r-1}(y^{p-1}-x^{p-1})^{p-r-1}\delta^r$ for any $0\le r\le p-1$.
  \item $d_1=x^{p^2-p}+y^{p^2-p}-x^{p-1}y^{*}-\delta\sum_{k=1}^{p-2}k\,x^{*}y^{k(p-1)-1}$.
  \item Let $i\ge p$, $j\ge1$ and let $a,b>0$ be integers such that $i=a(p-1)+b$ and $0<b\le p-1$. Then  $$x^iy^j=x^by^{*}+\langle\delta\rangle.$$
  \item Let $0<r<p-1$ and $j\ge r$. Then any monomial $x^iy^j$ can be written as an $\Fp$-linear combination
  $$x^iy^j=\sum_{s=0}^{rp-1}\lambda_s\,x^sy^*+\langle\delta^r\rangle.$$
  \qed
 \end{enumerate}
\end{Prop}

Let $r|p-1$. Define
$$\Omega^r:=\left\{x^iy^j\Bigg\vert\begin{aligned}&0\le i\le rp-1\\&0\le j\le p^2-p+r-1\end{aligned}\right\}\cup\left\{x^iy^j\Bigg\vert\begin{aligned}&rp\le i\le p^2-p-1\\&0\le j\le r-1\end{aligned}\right\}.$$
Notice that $\lvert\Omega^r\rvert=rp(p^2-1)=\lvert L^r\rvert$. 
As suggested by the definition above, it is useful to imagine $\Omega^r$ as consisting of two blocks: a first block which contains all monomials which divide $x^{rp-1}y^{p^2-p+r-1}$ and a second block with all monomials $x^iy^j$ for $rp\le i\le p^2-p-1$, $0\le j\le r-1$. This second block only appears if $r<p-1$.

\begin{Prop}\label{baseL}
 $\Omega^r$ is a $P^{L^r}$-basis for $P$.
\end{Prop}

\begin{proof}
 It is enough to prove that the monomials in $\Omega^r$ generate $P$ as a $P^{L^r}$-module. We prove this by induction on the degree. Let $f$ be a polynomial of degree $n$ and write $n=a(p^2-p)+b$ with $a,b\ge0$ and $0\le b<p^2-p$. Then it is clear that we can write $f=\lambda\, d_1^ax^b+yf'$ for some $\lambda\in\Fp$, $f'\in P$. By induction, we can write $f'=\sum h_i\alpha_i$ for some $h_i\in P^{L^r}$, $\alpha_i\in\Omega^r$. It is clear that for any $\alpha\in\Omega^r$ we have $y\alpha\in\Omega^r$, except in the following two cases:
 \begin{enumerate}
  \item $\alpha=x^iy^{p^2-p+r-1}$, $0\le i\le rp-1$.
  \par\noindent
  In this case, formula (3) in \ref{formules} tells us that $y\alpha=x^iy^{p^2-p+r}=d_1x^iy^r-\delta^rh$ for some $h\in P$. By induction, $h$ is in $P^{L^r}\langle\Omega^r\rangle$ and so is $y\alpha$.
  \item $\alpha=x^iy^{r-1}$, $r<p-1$, $rp\le i\le p^2-p-1$.\par\noindent
  In this case, we use \ref{formules}(6) to write the monomial $x^iy^r$ as 
  $$x^iy^r=\sum_{s=0}^{rp-1}\lambda_s\,x^sy^*+\delta^rh.$$ Then, we can apply induction on $h$ to finish the proof.
 \end{enumerate}
\end{proof}

The case $r=p-1$ of this result was known since the work of Campbell et al.\ (\cite{Campbell}) who computed a basis for 
the coinvariants of the group $GL_n(\Fp)$ for any $n$.

The case of the reducible groups $U_{r,s}$ is similar but easier.

\begin{Prop}
 For any $r,s|p-1$ the monomials $$\Gamma_{r,s}:=\{x^iy^j\,|\, 0\le i\le r-1,\, 0\le j\le ps-1\}$$ form a basis of $P$ over $P^{U_{r,s}}$.
\end{Prop}

\begin{proof}
 Again, it is enough to show that these monomials generate $P$ over $P^{U_{r,s}}$. Recall from \ref{invariantsU} that 
the ring of invariants of $U_{r,s}$ is generated by $x^r$ and $\rho=(y^p-x^{p-1}y)^s$. If $f$ is any polynomial, we can 
write $f=\lambda\rho^ay^b+xh$ for some $a\ge0$, $0\le b<ps$, $h\in P$. By induction on the degree, the polynomial 
$h$ is in the span of $\Gamma_{r,s}$. Then, for any $\alpha\in\Gamma_{r,s}$ we have $x\alpha\in\Gamma_{r,s}$, except for 
the case of $\alpha=x^{r-1}y^j$ for $0\le j<ps$, but in this case $x\alpha=x^ry^j\in 
P^{U_{r,s}}\langle\Gamma_{r,s}\rangle$.
\end{proof}

In all these cases, the ring of coinvariants is a Poincar{\'e} duality algebra. The fundamental class $\mu$ is 
$x^{rp-1}y^{p^2-p+r-1}$ for the group $L^r$ and $x^{r-1}y^{ps-1}$ for the group $U_{r,s}$.

%
%
%
%
%
%
%
%

\section{Ideals of stable invariants}
In this section we compute the sequence of ideals of stable invariants of the reflection groups in $GL_2(\Fp)$. We start with the harder case of the groups $L^r$. Fist, we need to compute the invariants of $L^r$ acting on the coinvariant algebra $P_{L^r}$. The following lemma on binomial coefficients modulo $p$ will be used in the proofs in this section.

\begin{Lem}\label{lemabinomial}
Let $0\le i<p$, $1\le k<p$. Then $$\sum_{t=0}^{k-1}\binom{k(p-1)}{i+t(p-1)}\equiv(-1)^i\mod p.$$ 
\end{Lem}

\begin{proof}
 We need a well known fact on the computation of binomial coefficients modulo a prime (see \cite{Steenrod}, lemma 2.6). 
If $a=\sum a_ip^i$, $b=\sum b_ip^i$ are the base $p$ series for the integers $a$ and $b$, then $$\binom 
ab\equiv\prod\binom{a_i}{b_i}\mod p.$$ Using this fact we can proceed as follows. If $i<k$ then
 $$\begin{aligned}
  \sum_{t=0}^{k-1}\binom{k(p-1)}{i+t(p-1)}&=\sum_{t=0}^{i}\binom{k(p-1)}{i+t(p-1)}+
 \sum_{t=i+1}^{k-1}\binom{k(p-1)}{i+t(p-1)}\\
 &\equiv\sum_{t=0}^{i}\binom{k-1}{t}\binom{p-k}{i-t}+\sum_{t=i+1}^{k-1}\binom{k-1}{t-1}\binom{p-k}{p+i-t}.
  \end{aligned}
  $$
  The first term is equal to $\binom{p-1}i\equiv(-1)^i\,(p)$ and the second term vanishes because $p+i-t>p-k$. Notice that the same argument holds for $i\ge k$.
\end{proof}

\begin{Thm}\label{calculinvest}
 Let $A$ be the elements of positive degree in the coinvariant algebra $P_{L^r}$ and let $A^{L^r}$ denote the $\Fp$-vector space of $L^r$-invariant elements of $A$. Then
 \begin{enumerate}
  \item If $r>1$, then $A^{L^r}=0$.
  \item If $r=1$, then $A^{L^r}$ is generated by the following elements:
  \begin{enumerate}
   \item The fundamental class $\mu=x^{p-1}y^{p^2-p}$.
   \item The elements $\gamma_i$ for $i=2,\ldots,p-1$.
  \end{enumerate}
 \end{enumerate}
\end{Thm}

\begin{proof}
The case of $p=2$ can be easily worked out directly (see example 1.10 in \cite{Larry-Ring}), so we assume that $p>2$. 
In proposition \ref{baseL} we have found an $\Fp$-basis $\Omega^r$ for $P_{L^r}$ so an element in $P_{L^r}$ of degree $n$ can be uniquely written as $$f=\sum_{x^iy^j\in\Omega^r}^{i+j=n}\lambda_{ij}\,x^iy^j.$$ The matrices $D(\zeta^{\frac{p-1}r},1)$, $D(1,\zeta^{\frac{p-1}r})$, $D(\zeta,\zeta^{-1})\in L^r$ 
act diagonally on the elements of the basis $\Omega^r$. Hence, if $f$ is invariant then
the coefficients $\lambda_{ij}$ must vanish except when $i,j\equiv0\,(r)$ and $2i\equiv n\,(p-1)$. Lets call this 
``sparseness''.

Recall that we are considering the basis $\Omega^r$ as formed by two blocks. Then, 
we can distinguish two cases for an hypothetical invariant $f$, depending if $f$ contains monomials from both of the 
blocks (this is only possible if $r<p-1$) or it only contains monomials from the first block. 
Let us consider first the case in which there is no monomial from the second block of $\Omega^r$. Then,
$$f=\sum_{i=k_0}^{rp-1}\lambda_ix^iy^{n-i}$$
where $k_0=\max\{0,n-p^2+p-r+1\}$. In the extreme case $k_0=rp-1$, $f$ is a multiple of 
the fundamental class $\mu=x^{rp-1}y^{p^2-p+r-1}$ which can 
only be an invariant if $r=1$. In this case, this fundamentals class $\mu$ of the Poincar{\'e} duality algebra 
of coinvariants of $L^1$ is indeed invariant under the action of $L^1$ because any element $g\in L^1$ must 
transform the fundamental class $\mu$ into some multiple of itself $\chi(g)\mu$, where $\chi$ is a character of $L^1$. 
But $L^1=SL_2(\Fp)$ has no non trivial characters.  

Assume $k_0<rp-1$. Using \ref{formules}(3) we see that $y^k=0$ in $A$ whenever $k\ge p^2-p+r$. If we assume that $f$ is 
invariant under $\omega$, we get $$f\!=\!\omega f\!=\!\sum_{i=k_0}^{rp-1}\sum_{j=0}^i\binom 
ij\lambda_i\,x^jy^{n-j}\equiv 
\sum_{i=k_0}^{rp-1}\sum_{j=k_0}^i\binom ij\lambda_i\,x^jy^{n-j}=
\sum_{j=k_0}^{rp-1}\left(\sum_{i=j}^{rp-1}\binom ij\lambda_i\right)x^jy^{n-j}.$$ Since all monomials here belong to the 
basis $\Omega^r$, their coefficients must agree with the coefficients in $f$ and we get a system of linear equations for 
the coefficients $\lambda_i$: $$\sum_{j=i+1}^{rp-1}\binom ji\lambda_j=0,\quad i=k_0,\ldots,rp-2.$$  The last equation 
gives $\lambda_{rp-1}=0$ and we can prove recursively that $\lambda_i=0$ for $i>0$ by noticing that 
$i\equiv0\,(p)$ and $\lambda_i\neq0$ imply $i\equiv0\,(pr)$ which is not possible.

Hence, $f$ must be a multiple of $y^n$. 
If $y^n$ were invariant under $\psMatrix01{-1}0$ then $$y^n-x^n\in\langle d_1,\delta^r\rangle\subset\langle 
d_1,\delta\rangle.$$ If $n<p^2-p=\lvert d_1\rvert$, then this is not possible because $\delta$ is divisible by $x$. 
Assume $n\ge p^2-p$. This implies $n=p^2-p$. Using \ref{formules}(4)  we deduce 
$$2y^{p^2-p}-x^{p-1}y^{p^2-2p+1}\in\langle d_1,\delta\rangle$$ which is not possible since these two monomials are in 
$\Omega^1$ which is a basis of $P$ over $P^{L^1}=\Fp[d_1,\delta]$. 

The conclusion so far is that the only invariants in $A$ which do not involve monomials from the second block of 
$\Omega^r$ are the scalar multiples of $\mu=x^{p-1}y^{p^2-p}$ for $r=1$. Let us investigate now under which conditions 
a polynomial involving monomials from the second block of $\Omega^r$ can 
be $L^r$-invariant. By sparseness, the only monomial from the second block that can appear is $x^n$ for 
$n\equiv0\,(p-1)$. Assume that  $f=\sum_{i=0}^{rp-1}\lambda_i\,x^iy^{n-i}+x^n$ is $L^r$-invariant with $rp\le n\le 
p^2-p-1$. We must have $r<p-1$ and $n\equiv0\,(p-1)$. Hence, $n>rp$. Consider the invariance of $f$ under $\omega$. We 
have $\omega x^n=x^n+nx^{n-1}y+\cdots$ and if $r>1$, then \ref{formules}(6) shows that the term $x^{n-1}y$ cannot cancel 
with any other term. This implies that $n\equiv0\,(p)$ which is not possible because $n\equiv0\,(p-1)$ and $n<p^2-p$. 

Hence, from now on we only need to investigate the case of $r=1$. Assume there is an invariant of the form  
$$f=\sum_{i=0}^{p-1}\lambda_i\,x^iy^{n-i}+x^n\text{, for }p\le n<p^2-p.$$ 
By sparseness, we know that $n\equiv0\,(p-1)$ and 
$\lambda_i=0$ except whenever $2i\equiv0\,(p-1)$. Hence, 
$$f=\lambda_0\,y^n+\lambda_1\,x^{\frac{p-1}2}y^{n-\frac{p-1}2}+\lambda_2\,x^{p-1}y^{n-p+1}+x^n.$$ Let us investigate the 
action of $\omega$ on $x^n$. Let $n=k(p-1)$ for some $k=2,\ldots,p-1$. In the following computation we use proposition 
\ref{formules}(5) and lemma \ref{lemabinomial}:

$$
\begin{aligned}
\omega x^n&=y^n+x^n+\sum_{i=1}^{p-1}\tbinom ni\, x^iy^{*}+\sum_{s=1}^{p-1}\sum_{\genfrac{}{}{0pt}{}{i=p}{i\equiv s\,(p-1)}}^{n-1}\tbinom ni\,x^sy^{*}\\
&=y^n+x^n+\sum_{i=1}^{p-1}\tbinom ni\, x^iy^{*}+\sum_{s=1}^{p-2}\sum_{\genfrac{}{}{0pt}{}{i=p}{i\equiv s\,(p-1)}}^{n-1}\tbinom ni\,x^sy^{*}+\sum_{\genfrac{}{}{0pt}{}{i=p}{i\equiv 0\,(p-1)}}^{n-1}\tbinom ni\,x^{p-1}y^{*}\\
&=y^n+x^n+\sum_{i=1}^{p-1}\tbinom ni\, x^iy^{*}+\sum_{s=1}^{p-2}\sum_{t=1}^{k-1}\tbinom{n}{s+t(p-1)}\,x^sy^{*}+\sum_{t=1}^{k-2}\tbinom{n}{(t+1)(p-1)}\,x^{p-1}y^{*}\\
&=y^n+x^n+\sum_{i=1}^{p-2}\sum_{t=0}^{k-1}\tbinom{n}{i+t(p-1)}\,x^iy^{n-i}\\
&=y^n+x^n+\sum_{i=1}^{p-2}(-1)^ix^iy^{n-i}
\end{aligned}
$$

Thus, if $\omega f=f$ we can compare the coefficients of $y^n$ and $xy^{n-1}$ in $f$ and $\omega f$ to get $\lambda_1=0$, $\lambda_2=-1$ and so $f=\lambda_0y^n-x^{p-1}y^{n-p+1}+x^n$. If we compute the action of $\psMatrix01{-1}0$ on $f$ we see that if $f$ is invariant then $\lambda_0=1$. The conclusion is that the invariants of positive degree in $A$ for $r=1$ are contained in the $\Fp$-span of the polynomials $\gamma_i$ for $i=2,\ldots,p-1$.

To finish the proof we have to prove that the polynomials $\gamma_i$ are invariant elements in the coinvariant algebra $A$. This follows immediately from the computation of $\omega f$ above and the (easier) computation of the action of $\psMatrix01{-1}0$ on each $\gamma_i$.
\end{proof}

This last theorem settles the computation of the ideal of stable invariants of the groups $L^r$ for $r>1$. Since the ring of coinvariants of these groups has no invariants of positive degree, the ideal of stable invariants coincides with the ideal of ordinary invariants.

When $r=1$ the situation is very different. To compute the stable invariants we need to investigate the quotient $Q$ of the ring of coinvariants $P_{L^1}$ by the ideal of its invariants of positive degree, as computed in theorem \ref{calculinvest}. 

\begin{Prop}\label{basedos}Let $Q:=P/\langle d_1,\delta,\mu,\gamma_2,\cdots,\gamma_{p-1}\rangle$. Then
 \begin{enumerate} 
  \item $Q=P/\langle d_1,\delta,\gamma_2,x^{p-1}y^{2p-2}\rangle$.
  \item $Q$ has an $\Fp$-basis as follows
  $$\Theta=\left\{x^iy^j\Big\vert\,\genfrac{}{}{0pt}{}{0\le i\le p-1,\,0\le j\le 2p-2}{i+j\neq3p-3}\right\}\cup\left\{x^i\,\vert\, p\le i\le2p-3\right\}.$$
  \item $Q$ has no $L^1$-invariants of positive degree.
 \end{enumerate}
\end{Prop}

\begin{proof}
 To see that $\Theta$ generates $Q$, we will show that the monomials in $\Omega^1-\Theta$ vanish or become redundant in 
$Q$. Notice that $\gamma_2=x^{2(p-1)}+y^{2(p-1)}-x^{p-1}y^{p-1}$ hence $x^{2(p-1)}=x^{p-1}y^{p-1}-y^{2(p-1)}$ in $Q$. On 
the other hand, $x\gamma_2=x^{2p-1}+\langle d_1,\delta\rangle$ by \ref{formules}(5) and so $x^{2p-1}=0$ in $Q$. By the 
same argument, $y^{2p-1}=0$ in $Q$. Finally, $\gamma_3=x^{3(p-1)}-x^{p-1}y^{2(p-1)}+y^{3(p-1)}$ hence 
$x^{p-1}y^{2(p-1)}=0$ in $Q$ and we have seen that $\Theta$ generates $Q$ and this also shows part (1) of the 
proposition.
 
 Let us see now that the elements in $\Theta$ are linearly independent in $Q$. Since $\Theta\subset\Omega^1$ and $\Omega^1$ is a basis of $P/\langle d_1,\delta\rangle$, it is enough to show that a linear combination $$f=\sum_{i=n-2p+2}^{p-1}\lambda_i\,x^iy^{n-i}$$ in degree $n$ with $2p-1\le n<3p-3$ such that $f+g\gamma_2\in\langle d_1,\delta\rangle$ for some $g$ of positive degree must have all coefficients $\lambda_i=0$, $n-2p+2\le i\le p-1$. Let $g=\sum_{i=0}^k\mu_i\,x^iy^{k-i}$ for $k=n-2p+2$, $1\le k\le p-2$. From \ref{formules}(5), we get the following identities modulo $\langle d_1,\delta\rangle$:
 $$
 \begin{aligned}
f+g\gamma_2&=f+\sum_{i=0}^k\mu_i\,x^{i+2p-2}y^*-\sum_{i=0}^k \mu_i\,x^{i+p-1}y^*+\sum_{i=0}^k \mu_i\,x^{i}y^*\\
&=f+\mu_kx^n+\sum_{i=0}^{k-1}\mu_i\,x^{i+2p-2}y^*-\mu_0x^{p-1}y^*-\sum_{i=1}^k \mu_i\,x^{i+p-1}y^*+\sum_{i=0}^k \mu_i\,x^{i}y^*\\
&=f+\mu_kx^n+\sum_{i=0}^{k-1}\mu_i\,x^iy^{n-i}\\
&=\sum_{i=k}^{p-1}\lambda_i\,x^iy^{n-i}+\mu_k\,x^{k+2p-2}+\sum_{i=0}^{k-1}\mu_i\,x^iy^{n-i}.
 \end{aligned}
 $$
 All monomials in the last term above are in $\Omega^1$ and so they are linearly independent modulo $\langle c_ 0,\delta\rangle$. Hence, $f+g\gamma_2\in\langle d_1,\delta\rangle$ implies $\lambda_i=0$ for $k\le i\le p-1$. 
 
 Notice that if $p=3$ we have $x^{p-1}y^{2p-2}\in\langle d_1,\delta,\gamma_2\rangle$ and so $$Q=P/\langle d_1,\delta,\gamma_2\rangle\text{ for $p=3$}.$$
 
 To prove that there are no $L^1$-invariants of positive degree in $Q$ we proceed in a similar way as in the proof of \ref{calculinvest}. Suppose there is an invariant $f=\sum_{i=k}^{p-1}\lambda_i\,x^iy^{n-i}$ in degree $n$ for $2p-2<n<3p-3$. Then
 $$
   \omega f=\sum_{i=k}^{p-1}\sum_{s=0}^i\binom is\lambda_i\,x^sy^{n-s}\equiv
\sum_{s=k}^{p-1}\left(\sum_{i=s}^{p-1}\binom is\lambda_i\right)\,x^sy^{n-s}.
  $$
Then, as in the proof of \ref{calculinvest}, $\omega f=f$ in $Q$ implies $\lambda_i=0$ for $k+1\le i\le p-1$ and $f$ is 
a scalar multiple of the monomial $x^ky^{2p-2}$ which is not invariant under $D(\zeta,\zeta^{-1})$. 
\end{proof}

The computation of the stable invariants of the groups $L^r$ is done:

\begin{Thm}\label{stableL}
The sequence of ideals of stable invariants of the reflection groups $L^r$ is as follows:
\begin{enumerate}
 \item If $r>1$ then $J_\infty(L^r)=J_1(L^r)=\langle d_1,\delta^r\rangle$.
 \item $J_\infty(L^1)=J_2(L^1)=\langle d_1,\delta,\gamma_2,x^{p-1}y^{2p-2}\rangle\supsetneq J_1(L^1)=\langle d_1,\delta\rangle$.\qed
\end{enumerate}
\end{Thm}

Let us consider now the decomposable groups $U_{r,s}$ for $r,s|p-1$. In proposition \ref{invariantsU} we have computed the ideal of ordinary invariants of these groups
$$J_1( U_{r,s})=\langle x^r,(y^p-x^{p-1}y)^s\rangle=\langle x^r,y^{sp}\rangle.$$

\begin{Thm}\label{stableU}
The sequence of ideals of stable invariants of the reflection groups $U_{r,s}$ is as follows: 
\begin{enumerate}
 \item If $r>1$ then $J_\infty(U_{r,s})=J_1(U_{r,s})=\langle x^r,y^{sp}\rangle$.
 \item $J_\infty(U_{1,s})=J_2(U_{1,s})=\langle x,y^s\rangle\supsetneq J_1(U_{1,s})=\langle x,y^{sp}\rangle$.
\end{enumerate}
\end{Thm}

\begin{proof}
It is obvious that the ring of coinvariants has a monomial basis consisting of the divisors of $x^{r-1}y^{sp-1}$. Invariance under the matrices $D(\zeta^{(p-1)/r},1)$ and $D(1,\zeta^{(p-1)/s})$ implies that an invariant of positive degree must be a multiple of $y^{si}$ for $i=1,\ldots,p-1$. If $r>1$ none of these monomials is invariant under $\omega'$ while all of them are invariant if $r=1$. This computes $J_2(U_{r,s})$ and the theorem follows.
\end{proof}

Some particular cases of this theorem were considered in \cite{Larry-Ring}.

%
%
%
%
%
%
%
%

\section{Ideals of generalized invariants}\label{geninv}
In this section we compute the ideal of generalized invariants for any set of reflections $\mathcal{S}\subset GL_2(\Fp)$. The computation of the ideal of generalized invariants is trivial in those cases in which the stable invariants of the group $G(\mathcal{S})$ coincide with the ordinary invariants of $G(\mathcal{S})$ and this happens, in particular, if the order of $G(\mathcal{S})$ is prime to $p$. Hence, theorems \ref{grups}, \ref{stableL} and \ref{stableU} imply that in the case of rank two the computation of the generalized invariants is only relevant for those sets of reflections $\mathcal{S}$ which generate one of the groups $L^1$, $U_{1,s}$ for $s|p-1$. 

\begin{Thm}\label{genU}
 Let $\mathcal{S}$ be a set of reflections generating $U_{r,s}$ for some $r,s|p-1$. Then 
 $$
 I(\mathcal{S})=\begin{cases}
 \langle x^r,y^{sp}\rangle&\text{ if $r>1$ or $\mathcal{S}$ contains an element of order $p$;}\\                
 \langle x,y^s\rangle&\text{ otherwise.}
 \end{cases}
$$
\end{Thm}

\begin{proof}
 We can assume $r=1$. We have $$\langle x,y^{sp}\rangle=J_1(U_{1,s})\subseteq I(\mathcal{S})\subseteq J_\infty(U_{1,s})=\langle x,y^s\rangle.$$
 Hence, $I(\mathcal{S})=\langle x,y^k\rangle$ for some $k$ with $s\le k\le sp$. This value of $k$ is characterized by the fact that $$\Delta_\alpha(y^k)\in\langle x\rangle\text{ for any }\alpha\in\mathcal{S}.$$ 
 It is not difficult to compute the values of $\Delta_\alpha(y^k)$ modulo $x$ for any reflection $\alpha\in U_{1,s}$:
 $$\Delta_\alpha(y^k)=
 \begin{cases}
  \lambda ky^{k-1}+\langle x\rangle&\text{ for } \alpha=\psMatrix1\lambda01,\,\lambda\neq0 \\
  \frac{\beta^k-1}{\beta-1}\,y^{k-1}+\langle x\rangle&\text{ for } \alpha=\psMatrix1\lambda0\beta,\,\beta\neq0,1
 \end{cases}
 $$ 
 
 If $s=1$ then we are dealing with the cyclic group of order $p$ and $\mathcal{S}$ contains only powers of $\omega'$. In this case it is easy to conclude that $I(\mathcal{S})=\langle x,y^p\rangle$. Assume $s>1$ and let $$\mathcal{S}=\left\{\Matrix1{\lambda_i}0{\beta_i},\,i=1,\ldots,n\right\}.$$
 We have that $s$ is the l.c.m.\ of the orders of all $\beta_i$ and so $s|k$ by the computation above. If $\mathcal{S}$ 
also contains a reflection of order $p$ then $p|k$ and $k=ps$. The theorem follows. 
\end{proof}

In the case of the groups $L^r$ we obtain that the generalized invariants coincide with the ordinary invariants.

\begin{Thm}\label{genL}
Let $\mathcal{S}$ be a set of reflections generating $L^r$. Then $I(\mathcal{S})=J_1(L^r)=\langle d_1,\delta^r\rangle$.
\end{Thm}

\begin{proof}
 As said before, we just need to consider the case of $L^1$. We know that 
 $$\langle\delta,d_1\rangle=J_1(L^1)\subseteq I(\mathcal{S})\subseteq 
J_\infty(L^1)=\langle\delta,\gamma_2,x^{p-1}y^{2p-2},d_1\rangle.$$
 We also know that $I(\mathcal{S})$ is generated by a regular sequence of two polynomials, hence for each generating set 
of reflections $\mathcal{S}$ we can find a polynomial $f(\mathcal{S})$ such that 
$$I(\mathcal{S})=\langle\delta,f(\mathcal{S})\rangle.$$It is clear that $\lvert f(\mathcal{S})\rvert\le\lvert 
d_1\rvert=p(p-1)$ and the theorem is proven if we show that $\lvert f(\mathcal{S})\rvert=p(p-1)$. On the other hand, 
$f(\mathcal{S})\in\langle\delta,\gamma_2,x^{p-1}y^{2p-2},d_1\rangle-\langle\delta\rangle$ hence $\lvert 
f(\mathcal{S})\rvert\ge2p-2$. Also, $f(\mathcal{S})$ has the property that 
 $$\Delta_\alpha(f(\mathcal{S}))\in\langle\delta\rangle\text{ for any }\alpha\in\mathcal{S}.$$
 
 Let $f$ be any polynomial of degree $n$ with $2(p-1)\le n<p^2-p$.
 Using the base $\Omega^1$ we can write $$f=\lambda_n\,x^n+\sum_{i=0}^{p-1}\lambda_i\,x^iy^{n-i}+\langle\delta\rangle.$$ 
Let us investigate the action of the operator $\Delta$ associated to the reflection $\omega$ on $f$ modulo 
$\langle\delta\rangle$. We have
 $$
 \begin{aligned}
  \Delta(f)&=\lambda_n\sum_{j=0}^{n-1}\binom njx^jy^{*}+\sum_{i=1}^{p-1}\lambda_i\sum_{j=0}^{i-1}\binom ijx^jy^{*}+\langle\delta\rangle\\
  &=\lambda_n\sum_{j=0}^{n-1}\binom njx^jy^{*}+\sum_{j=0}^{p-2}\sum_{i=j+1}^{p-1}\binom ij\lambda_i\,x^jy^{*}+\langle\delta\rangle
 \end{aligned}
$$
 Checking the coefficient of $x^{n-1}$ in the formulas above we get that if $\Delta(f)\in\langle\delta\rangle$ then $n\lambda_n=0$.
 
 Let us first consider the case in which $\lambda_n=0$. Then, $\Delta(f)\in\langle\delta\rangle$ yields a system of linear equations
 $$\sum_{i=j+1}^{p-1}\binom ij\lambda_i=0,\quad j=0,\ldots,p-2$$ which has only the trivial solution $\lambda_i=0$, $i=1,\ldots,p-1$. Hence, $f=\lambda_0\,y^n$.
 
 If $\lambda_n\neq0$ then $n=mp$ with $2\le m<p-1$. Notice that
 $$
 \begin{aligned}
 \Delta(x^{mp})&=\frac{(x^p+y^p)^m-x^{mp}}y=\sum_{i=0}^{m-1}\binom mix^{ip}y^*\\
 &=y^{mp-1}+\sum_{i=1}^{m-1}\binom mix^{ip}y^*=y^{mp-1}+\sum_{i=1}^{m-1}\binom mix^{i}y^*+\langle\delta\rangle.
 \end{aligned}
 $$
 Hence, taking $\lambda_n=1$ we have
 $$
 \Delta(f)=\sum_{i=0}^{m-1}\binom mix^{i}y^*+\sum_{i=0}^{p-2}
 \left(\sum_{j=i+1}^{p-1}\binom ji\lambda_j\right)\,x^iy^*+
 \langle\delta\rangle
 $$
 and $\Delta(f)\in\langle\delta\rangle$ yields a system of linear equations like before
 $$\sum_{i=j+1}^{p-1}\binom ij\lambda_i=
 \begin{cases}
  -\binom mj&j=0,\ldots,m-1\\0&j=m,\ldots,p-2.
 \end{cases}
$$
 It is not difficult to find the unique solution of these equations which is $\lambda_m=-1$, $\lambda_i=0$, $0<i\neq m$.
 
 The conclusion so far is that the only polynomials modulo $\langle\delta\rangle$ in degree $n$ with $2(p-1)\le 
n<p^2-p$ such that $\Delta(f)\in\langle\delta\rangle$ are the scalar multiples of $y^n$ and 
 $$
  h_m:=x^{mp}+\lambda y^{mp}-x^my^{m(p-1)}
$$
for $2\le m<p-1$ and any $\lambda$.

Recall that $\overline{\Delta}$ is the operator associated to the reflection $\omega'$. 
Checking the coefficient of $x^{n-1}$, it is clear that 
$\overline{\Delta}(y^n)\not\in\langle\delta\rangle$. Let us investigate if $\overline{\Delta}(h_m)\in\langle\delta\rangle$ for some value of $m$. We have
$$
\begin{aligned}
\overline{\Delta}(h_m)=\lambda x^{mp-1}+&\lambda mx^{p-1}y^*+\lambda\sum_{j=1}^{m-2}\binom{m}{j+1}x^jy^*\\-&
\sum_{r=0}^{m(p-1)-1}\binom{m(p-1)}{r+1}x^{m+r}y^*+\langle\delta\rangle.
\end{aligned}
$$
If we use \ref{formules}(5) to write $\overline{\Delta}(h_m)$ in the basis $\Omega^1$ we see that the coefficient of $x^my^*$ is equal to 
$$\sum_{j=0}^{m-1}\binom{m(p-1)}{1+j(p-1)}\equiv-1\;(p)$$
(by lemma \ref{lemabinomial}) hence $\overline{\Delta}(h_m)\notin\langle\delta\rangle$. 

All the arguments so far prove that if $\mathcal{S}$ contains the reflections $\omega$ and $\omega'$, then we can choose $f(\mathcal{S})=d_1$ and the theorem is proven in this case. 
The general case follows easily. Since $\mathcal{S}$ generates $L^1=SL_2(\Fp)$, up to a conjugation in $GL_2(\Fp)$ we can assume that $\mathcal{S}$ contains $\omega$ and some other reflection $\alpha=\psMatrix1\lambda01$. Notice that $\omega'$ is equal to some power of $\alpha$.  Then, lemma 2.4 in \cite{Larry-Ring} gives
$$I(\mathcal{S})=I(\mathcal{S}\cup\{\omega'\})=\langle\delta,d_1\rangle.$$
\end{proof}

In these results we find examples of the following two phenomena (already noted in \cite{Kac-Pet} and \cite{Larry-Ring}):
\begin{enumerate}
 \item The ideal of generalized invariants $I(\mathcal{S})$ depends on the set of reflections $\mathcal{S}$, not just on the reflection group $G(\mathcal{S})$ they generate.
 \item The ideal of generalized invariants $I(\mathcal{S})$ can be strictly smaller than the ideal of stable invariants $J_\infty(G(\mathcal{S}))$.
\end{enumerate}

In particular, $\gamma_2$ is a stable invariant for $L^1$ but it is not a generalized invariant for any set of generating reflections of $L^1$. We finish this section providing an explicit sequence of length $2p-2$ of $\Delta$-operators which does not vanish on $\gamma_2$. The following result contains also some computations of these operators which may be useful elsewhere.

\begin{Prop}\label{operadorsD}
 Let $\Delta$ and $\overline\Delta$ be as before. Then
 \begin{enumerate}
  \item The following recurrence formula holds for $i\ge1$, $j\ge0$:
  $$\Delta^i(x^{i+j})=\sum_{k=0}^j\binom{i+j}{i+k-1}\big(\Delta^{i-1}(x^{i+k-1})\big)y^*$$
  \item $\Delta^i(x^i)=i!$ for $i\ge0$.
  \item $\Delta^i(x^{i+1})=(i+1)!(x+iy/2)$ for $i\ge0$.
  \item $\Delta^i(x^pz)=iy^{p-1}\Delta^{i-1}(z)+iy^p\Delta^i(z)+x^p\Delta^i(z)$ for $i\ge0$ and any $z\in P$.
  \item $\Delta^{p-2}(x^{2p-2})=(p-2)!(x^p+y^p-2xy^{p-1})$.
  \item $\Delta^{p-2}(\gamma_2)=(p-2)!(x^p-xy^{p-1})$.
  \item $\Delta\overline\Delta^{p-1}\Delta^{p-2}(\gamma_2)\neq0$. Hence, $\gamma_2$ is not a generalized invariant for the reflections $\omega$ and $\omega'$.
 \end{enumerate}
\end{Prop}

\begin{proof}
 The proof of each of these formulas is easy and can be left as an exercise.
 To prove (1), decompose $\Delta^i=\Delta^{i-1}\Delta$. (2) and (3) follow immediately from (1) for $j=0$ and $j=1$, 
respectively, plus induction. (4) is proven by induction on $i$. To prove (5), write 
$\Delta^{p-2}x^{2p-2}=\Delta^{p-2}(x^px^{p-2})$ and use (3) and (4). (6) follows from (5) and (3). (7) is then obvious 
from (6) and (2).
\end{proof}

%
%
%
%
%
%
%
%

\section{Modular rank two reflection groups in topology}
Real reflection groups appear in topology as Weyl groups of compact Lie groups and reflection groups over $\Fp$ appear 
as mod $p$ reductions of these Weyl groups. Then, the invariant theory of these groups is crucial in understanding the 
cohomology of the compact Lie groups and the cohomology of their classifying spaces. Actually, the theory of 
generalized invariants was invented (see \cite{Demazure} and \cite{Kac-Lie}) in order to understand the mod $p$ 
cohomology of the compact connected Lie groups. Furthermore, the family of compact Lie groups has been extended in 
several directions (Kac-Moody groups \cite{Nitu}, $p$-compact groups \cite{Dwyer-Wilkerson}, $p$-local finite groups 
\cite{BLO}) and in each case, the invariant theory of reflection groups over $\Fp$ plays an important role in 
understanding the properties of these objects, mainly its cohomology.

Let $K$ be a compact connected Lie group or, more in general, a compact form of a Kac-Moody group (in the sense of 
\cite{Nitu}), with maximal torus $T$. Let $M$ be the weight lattice and let $\mathcal{S}=\{s_1,\ldots,s_n\}$ be the 
reflections with respect to a set of simple roots. $\mathcal{S}$ generates the Weyl group $W$ of $K$ and the symmetric 
algebra $S(M)$ can be identified to the cohomology of the classifying space $BT$. Choose a prime $p$ and denote $M_p$ 
the mod $p$ reduction of $M$ and $W_p\subset GL_n(\Fp)$ the mod $p$ reduction of the Weyl group $W$, acting on $M_p$. 
There is a map $$\psi:S(M_p)=H^*(BT;\Fp)\longrightarrow H^*(K/T;\Fp).$$ Then, the kernel of this map coincides with the 
ideal of generalized invariants $I(\mathcal{S})$: $$\ker\psi=I(\mathcal{S}).$$ This fundamental result is due to Kac 
(\cite{Kac-Lie}, see also \cite{Larry-Lie} and Theorem 6.1 in \cite{Nitu}).

Hence, some of the groups that we have investigated in the preceding sections should appear as mod $p$ reductions of 
Weyl groups. If the rank is two, then the Weyl group is either cyclic of order two or dihedral. It is easy to 
check that among the groups that we have considered in this paper (i.e.\ reflection groups in $GL_2(\Fp)$ of order 
divisible by $p$), the only dihedral groups are $GL_2(F_2)$ for $p=2$, and $U_{1,2}$, $U_{2,1}$, $U_{2,2}$ for any 
prime. Hence, these and the cyclic group $U_{1,1}$ for $p=2$, are the only groups of order divisible by $p$ that may 
appear as mod $p$ reductions of Weyl groups.

It is not difficult to compute $W_p$ for all compact connected Lie groups of rank two. We get groups or order divisible 
by $p$ in the following cases:

\begin{itemize}
 \item $SU(3)$ gives $W_2=GL_2(F_2)$ and $W_3=U_{2,1}$ with 
$\mathcal{S}=\left\{\psMatrix{-1}101\psMatrix{-1}001\right\}$.
\item $Sp(2)$ gives $W_2=U_{1,1}$.
\item $G_2$ gives $W_2=GL_2(F_2)$ and $W_3=U_{2,2}$ with 
$\mathcal{S}=\left\{\psMatrix{1}10{-1}\psMatrix{-1}001\right\}$.
\item $PSp(2)$ gives $W_2=GL_2(F_2)$.
\item $PSU(3)$ gives $W_2=GL_2(F_2)$ and $W_3=U_{1,2}$ 
with $\mathcal{S}=\left\{\psMatrix{1}10{-1}\psMatrix{1}00{-1}\right\}$. The case of this group at the prime 3 is the 
only case in which there is a discrepancy between ordinary and generalized invariants. We have $J_1=\langle 
x,y^6\rangle$ while $I(\mathcal{S})=\langle x,y^2\rangle$.
\item $U(2)$ gives $W_2=U_{1,1}$. 
\end{itemize}

We obtain all possible reflection groups, but only for very small primes. If we extend our scope to rank two Kac-Moody 
groups, then these reflection groups appear as mod $p$ Weyl groups $W_p$ for all primes. For any pair of 
positive integers $a,b$ with $ab>4$, let us denote $K(a,b)$ the compact form of the simply connected Kac-Moody group 
associated to the Cartan Matrix $\psMatrix2{-a}{-b}2$. The groups $K(a,b)$, their Weyl groups, and the cohomology of 
the classifying spaces $BK(a,b)$ were investigated in \cite{ABKS}. From Propositions 7.1, 6.2 and 3.2 in \cite{ABKS} 
one deduces the following:

\begin{itemize}
 \item $K(a,b)$ gives $W_p=U_{2,2}$ for any prime $p>2$ which divides $a$ but not $b$.
 \item $K(a,b)$ gives $W_p=U_{2,1}$ for any prime $p>2$ which divides $4-ab$.
 \item $PK(a,b)$ gives $W_p=U_{1,2}$ for any prime $p>2$ such that $p\vert4-ab$ and $p^2\!\!\not\vert4-ab$.
\end{itemize}

The identifications of these Weyl groups and their invariant theory was crucial in the cohomological computations in 
\cite{ABKS}.

The rank two reflection groups which are not dihedral groups (i.e.\ the groups $L^r$ and $U_{a,b}$ for $a>2$ or $b>2$) 
could appear as mod $p$ reductions $W_p$ of Weyl groups of $p$-compact groups. The classification of all $p$-compact 
groups is available (see \cite{Viru} and \cite{AG2}) and if we check all the cases of rank two we see that there is 
only one $p$-compact group which realizes a modular rank two reflection group not considered above. The relevant prime 
is $p=3$ and the reflection group is $GL_2(\F_3)$. It turns out that this reflection group is the group number 12 in 
the list of complex reflection groups (\cite{She-Todd}) and it can be lifted to a 3-adic reflection group $W$. In 1984 
Zabrodsky proved the existence of a 3-compact group of rank two $X_{12}$ having $W$ as Weyl group (\cite{Zab}, see also 
\cite{Jaume}). This completes the list of reflection groups in $GL_2(\Fp)$ of order divisible by $p$ that have appeared 
(so far) as mod $p$ Weyl groups of compact Lie groups or other topological groups generalizing compact Lie groups.



\begin{thebibliography}{99}

\bibitem{Jaume}
J.~Aguad{\'e}, \emph{Constructing modular classifying spaces}. Israel J.\ Math.\ 66, nos 1--3 (1989), 23--40.

\bibitem{AR}
J.~Aguad{\'e}, A.~Ruiz, \emph{Cohomology of Kac-Moody groups over a finite field}. Algebr.\ Geom.\ Topol.\ 13 (2013), 
no.\ 4, 2207--2238. 

\bibitem{ABKS}
J.~Aguad{\'e}, C.~Broto, N.~Kitchloo, L.~Saumell, \emph{Cohomology of classifying spaces of central
quotients of rank two Kac-Moody groups.} J.\ Math.\ Kyoto Univ.\ 45 (2005), no.\ 3, 449--488.

\bibitem{Viru}
K.\ K.\ S.\ Andersen, J.\ Grodal, J.\ M.\  M{\o}ller, A.\ Viruel, \emph{The classification of $p$-compact groups for 
$p$ odd.} Ann.\ of Math.\ (2) 167 (2008), no.\ 1, 95--210. 

\bibitem{AG2}
K.\ K.\ S.\ Andersen, J.\ Grodal, \emph{The classification of 2-compact groups.}
J.\ Amer.\ Math.\ Soc.\ 22 (2009), no.\ 2, 387--436. 

\bibitem{Benson}
D.\ J.\ Benson, \emph{Polynomial Invariants of Finite Groups}, LMS Lecture Notes Series, vol.\ 190, Cambridge Univ.\ 
Press, Cambridge, 1993.

\bibitem{BLO}
C.\ Broto, R.\ Levi, B.\ Oliver, \emph{The homotopy theory of fusion systems.}
J.\ Amer.\ Math.\ Soc.\ 16 (2003), no.\ 4, 779--856. 

\bibitem{Campbell}
H.\ E.\ A.~Campbell, I.\ P.~Hughes, R.\ J.~Shank, D.\ L.~Wehlau, \emph{Bases for rings of coinvariants}, Trans.\ Groups 
1, No.\ 4 (1996), 307--336.

\bibitem{Demazure}
M.~Demazure, \emph{Invariants sym{\'e}triques entiers des groupes de Weyl et torsion}. Invent.\ Math.\ 21 (1973), 
287--301. 

\bibitem{Dwyer-Wilkerson}
W.\ G.\ Dwyer, C.\ W.\ Wilkerson, \emph{Homotopy fixed-point methods for Lie groups and finite loop spaces.}
Ann.\ of Math.\ (2) 139 (1994), no.\ 2, 395--442. 

\bibitem{Kac-Lie}
V.\ G.~Kac, \emph{Torsion in cohomology of compact Lie groups and Chow rings of reductive algebraic groups}. Invent.\ 
Math.\ 80 (1985), no. 1, 69--79. 

\bibitem{Kac-Pet}
V.\ G.~Kac, D.\ H.~Peterson, \emph{Generalized invariants of groups generated by reflections}. Geometry today (Rome, 
1984), 231--249, Progr.\ Math., 60, Birkh{\"a}user Boston, Boston, MA, 1985.

\bibitem{Kemper}
G.~Kemper, \emph{Calculating invariant rings of finite groups over arbitrary fields}, J.\ Symb.\ Comp.\ 21 (1996), 
351--366.

\bibitem{Kemper-Malle}
G.\ Kemper, G.\ Malle, \emph{The finite irreducible linear groups with polynomial ring of
invariants}. Transform.\ Groups 2 (1997), no.\ 1, 57--89.

\bibitem{Nitu}
N.\ Kitchloo, \emph{On the Topology of Kac-Moody groups}. Math.\ Z.\ 276 (2014), 727--756.

\bibitem{Larry-Ring}
F.~Neumann, M.\ D.~Neusel, L.~Smith, \emph{Rings of generalized and stable invariants of pseudoreflections and 
pseudoreflection groups}. J.~Alg.\ 182 (1996) 85--122. 

\bibitem{Larry-Lie}
F.~Neumann, M.\ D.~Neusel, L.~Smith, \emph{Rings of generalized and stable invariants and classifying spaces of compact 
Lie groups}. Higher homotopy structures in topology and mathematical physics (Poughkeepsie, NY, 1996), 267--285,
Contemp.\ Math., 227, Amer.\ Math.\ Soc., Providence, RI, 1999. 

\bibitem{She-Todd}
G.\ C.~Shephard, J.\ A.~Todd, \emph{Finite unitary reflection groups}, Canadian J.\ Math.\ 6 (1954), 274--304.

\bibitem{Larry-book}
L.~Smith, \emph{Polynomial Invariants of Finite Groups}. A K Peters, Wellesley, Mass., 1995.

\bibitem{Steenrod}
N.\ E.~Steenrod,  \emph{Cohomology operations}. Lectures by N.\ E.\ Steenrod written and revised by D.\ B.\ A.\ 
Epstein. Annals of Mathematics Studies, No.\ 50 Princeton University Press, Princeton, N.\ J.\ 1962.

\bibitem{Suzuki}
M.~Suzuki, \emph{Group theory. I.} Grundlehren der Mathematischen Wissenschaften, 247. Springer-Verlag, Berlin-New York, 
1982.

\bibitem{Zab}
A.\ Zabrodsky, \emph{On the realization of invariant subgroups of $\pi_*X$}. Trans.\ Amer.\ Math.\ Soc.\ 285 (1984), 
467--496.
\end{thebibliography}
\end{document}